\documentclass{article}[12pt]

\usepackage[
left=1.25in, right=1.25in]{geometry}
\usepackage[all]{xy}
\usepackage{graphicx}
\usepackage{indentfirst}
\usepackage{bm}
\usepackage{amsthm}
\usepackage{mathrsfs}
\usepackage{latexsym}
\usepackage{amsmath}
\usepackage{amssymb}
\usepackage{color}
\usepackage{hyperref}
\usepackage{microtype}

\begin{document}

\newtheorem{theorem}{Theorem}[section]
\newtheorem{main}{Main}[section]
\newtheorem{proposition}[theorem]{Proposition}
\newtheorem{corollary}[theorem]{Corollary}
\newtheorem{definition}[theorem]{Definition}
\newtheorem{lemma}[theorem]{Lemma}
\newtheorem{example}[theorem]{Example}
\newtheorem{remark}[theorem]{Remark}
\newtheorem{question}[theorem]{Question}
\newtheorem{conjecture}[theorem]{Conjecture}
\newtheorem{fact}[theorem]{Fact}
\newtheorem*{ac}{Acknowledgements}

\title{Young's inequality for locally compact quantum groups}
\author{Zhengwei Liu 
 \footnote{Department of Mathematics and Department of Physics, Harvard University, Cambridge, 02138, USA. Email: \href{zhengweiliu@fas.harvard.edu}{zhengweiliu@fas.harvard.edu}} 
\and
Simeng Wang 
\footnote{Laboratoire de Mathematiques, Universite de Franche -Comte, Besancon Cedex, 25030, France and Institute of Mathematics, Polish Academy of Sciences, ul. Sniadeckin 8, 00-956 Warszawa, Poland.
Email: \href{simeng.wang@univ-fcomte.fr}{simeng.wang@univ-fcomte.fr}}
\and
Jinsong Wu 
\footnote{ School of Mathematical Sciences, University of Science and Technology of China, Hefei, 230026, P.R.China.
Email: \href{wjsl@ustc.edu.cn}{wjsl@ustc.edu.cn}}
}
\date{}

\maketitle

\begin{abstract}
In this paper, we generalize Young's inequality for locally compact quantum groups and obtain some results for extremal pairs of Young's inequality and extremal functions of Hausdorff-Young inequality.
\end{abstract}

\section{INTRODUCTION}
A fundamental result in Group theory is Young's inequality which was first studied by Young \cite{Young} in 1912.  Let $G$ be a locally compact group with a modular function $\delta_0$. Suppose $\frac{1}{r}+1=\frac{1}{p}+\frac{1}{q}$, $p,q,r\in [1,\infty]$ and norms and convolution are defined with relative to a left Haar measure. Then Klein and Russo \cite{KlRu} formulate Young's inequality as
$$\|f*g\delta_0^{1/p'}\|_r\leq \|f\|_p\|g\|_q,$$
where $f\in L^p(G), g\in L^q(G)$ and $\frac{1}{p}+\frac{1}{p'}=1$.
Now one could ask following two natural questions:
\begin{itemize}
\item[1.] Is the coefficient $1$ of $\|f\|_p\|g\|_q$ the best constant for Young's inequality?
\item[2.] Are there extremal functions for Young's inequality exists? If so what are they?
\end{itemize}
For the first question, it is not true in general and Beckner\cite{Beck} proved a sharp Young's inequlity for convolution on $\mathbb{R}^n$:
$$\|f*g\|_r\leq (A_pA_qA_{r'})^n\|f\|_p\|g\|_q,$$
where $f\in L^p(\mathbb{R}^n)$, $g\in L^q(\mathbb{R}^n)$, $p,q,r\in [1,\infty]$, $\frac{1}{r}+1=\frac{1}{p}+\frac{1}{q}$, $\frac{1}{m}+\frac{1}{m'}=1$, $A_m=(m^{1/m}/{m'}^{1/m'})^{1/2}$.
In 1977, Fournier \cite{Four} proved that the coefficient $1$ is the best constant for Young's inequality for a unimodular locally compact group if and only if it contains a compact open subgroup.

For the second question, Fournier \cite{Four} proved for unimodular locally compact groups that if the best constant for Young's inequality is $1$, then $f$ and $g$ are a left translation and a right translation of a subcharacter respectively.
Beckner\cite{Beck} showed for $\mathbb{R}^n$ in which case the best constant is not $1$, that the extremal functions are Gaussian.
Klein and Russo \cite{KlRu} showed that the sharp Young's inequality for Heisenberg groups does not admit any extremal functions.

There are several ways to generalize locally compact groups.
One of them is Kac algebra which was elaborated independently by Enock and Schwartz \cite{EnSchwart}, and by Kac and Vajnermann in the 1970s.
In 2000, Kustermans and Vaes \cite{KuVaes} introduced the locally compact quantum group which is a generalization of Kac algebra.
Note that the subfactor \cite{Jon83} also can be viewed as a generalization of Kac algebra.
It was shown by Enock and Nest \cite{EnNe} that there is one to one correspondence between Kac algebras and irreducible depth-$2$ subfactors.
For the finite-index case, C. Jiang, Z. Liu and J. Wu \cite{JLW} proved Young's inequality for subfactors.

In this paper, our goal is to generalize Young's inequality for locally compact quantum groups. We show that
\begin{main}[Theorem \ref{younglcqg}]
Let $\mathbb{G}$ be a locally compact quantum group. For $1\leq p,q,r\leq 2$ with $\frac{1}{r}+1=\frac{1}{p}+\frac{1}{q}$, let $p'$ be such that $\frac{1}{p}+\frac{1}{p'}=1$.  Suppose $x\in L^p(\mathbb{G})$ and $y\in L^q(\mathbb{G})$. Then
$$\|x*\rho_{-i/p'}(y)\|_r\leq \|x\|_p\|y\|_q.$$
\end{main}
We refer to Section 2 for the notations.

When the scaling automorphism group $\tau$ of $\mathbb{G}$ is nontrivial, Young's inequality is not true for $1\leq p,q,r\leq \infty$ in general.
If the scaling automorphism group is trivial, we have the following theorem.

\begin{main}[Theorem \ref{youngkac}]
Let $\mathbb{G}$ be a locally compact quantum group whose scaling automorphism group is trivial.
Suppose $1\leq p,q,r\leq \infty $, $\frac{1}{p}+\frac{1}{q}=\frac{1}{r}$.
If $\varphi=\psi$, then for $x\in L^p(\mathbb{G})$ and $y\in L^q(\mathbb{G})$, we have
$$\|x*y\|_r \leq \|x\|_p\|y\|_q.$$
If $\varphi$ is tracial, then for $x\in L^p(\mathbb{G})$ and $y\in L^q(\mathbb{G})$, we have
$$\|x*y\delta^{-1/p'}\|_r \leq \|x\|_p\|y\|_q.$$
\end{main}

Note that the noncommutative $L^p$ space given here is taken with respect to the left Haar weight. So is the convolution.
One could define the convolution with respect to the right Haar weight when the noncommutative $L^p$ space is taken with respect to the right Haar weight.

We also give the definition of shifts of group-like projections and show that they are extremal element for the Hausdorff-Young inequality given in \cite{Cooney}.
Similar results for Young's inequality is also obtained.
But we are not sure that all the extremal elements for the Hausdorff-Young inequality are shifts of group-like projections.

The rest of the paper is organized as follows.
In Section 2 we give a brief introduction to locally compact quantum groups and noncommutative L$^p$ spaces.
In Section 3 we prove Young's inequality for locally compact quantum groups.
In Section 4 we investigate the properties of shifts of group-like projections and show that they are extremal functions for Hausdorff-Young inequality.
\section{PRELIMINARIES}
In this section we will recall the definition of locally compact quantum groups and noncommutative $L^p$ spaces.
Let $\mathcal{M}$ be a von Neumann algebra with a normal, semi-finite, faithful weight $\varphi$. Recall that
$$\mathfrak{N}_\varphi=\{x\in\mathcal{M}|\varphi(x^*x)<\infty\},
\quad\mathfrak{M}_\varphi=\mathfrak{N}_\varphi^*\mathfrak{N}_\varphi,$$
where $\mathfrak{M}_\varphi$ is a *-subalgebra of $\mathcal{M}$.
Denote by $\mathcal{H}_\varphi$ the Hilbert space obtained by completing $\mathfrak{N}_\varphi$.
The map $\Lambda_\varphi:\mathfrak{N}_\varphi\mapsto \mathcal{H}_\varphi$ is the inclusion map.
Denote by $\pi_\varphi$ the *-isomorphism of $\mathcal{M}$ on $\mathcal{H}_\varphi$ given by $\pi_\varphi(a)\Lambda_\varphi(b)=\Lambda_\varphi(ab)$ for any $a\in\mathcal{M}$, $b\in\mathfrak{N}_\varphi$.
The triplet $(\pi_\varphi,\mathcal{H}_\varphi,\Lambda_\varphi)$ is the semi-cyclic representation of $\mathcal{M}$. We denote by $\nabla_\varphi$ the modular operator for $\varphi$, $\sigma_t^\varphi$ the modular automorphism group for $\varphi$, $t\in\mathbb{R}$, $J_\varphi$ the conjugate unitary on $\mathcal{H}_\varphi$.

A locally compact quantum group $\mathbb{G}=(\mathcal{M},\Delta,\varphi,\psi)$ consists of
\begin{itemize}
\item[(1)] a von Neumann algebra $\mathcal{M}$,
\item[(2)] a normal, unital, *-homomorphism $\Delta: \mathcal{M}\to\mathcal{M}\overline{\otimes}\mathcal{M}$ such that
$(\Delta\otimes\iota)\circ\Delta=(\iota\otimes\Delta)\circ\Delta$,
\item[(3)] a normal, semi-finite, faithful weight $\varphi$ such that
$(\iota\otimes\varphi)\Delta(x)=\varphi(x) 1$, $\forall x\in \mathfrak{M}_\varphi^+$;\\
a normal, semi-finite, faithful weight $\psi$ such that
$(\psi\otimes \iota)\Delta(x)=\psi(x)1$, $\forall x\in \mathfrak{M}_\psi^+$,
\end{itemize}
where $\overline{\otimes}$ denotes the von Neumann algebra tensor product and $\iota$ denotes the identity map. The normal, unital, *-homomorphism $\Delta$ is a comultiplication of $\mathcal{M}$, $\varphi$ is the left Haar weight, and $\psi$ is the right Haar weight.

We assume that $\mathcal{M}$ acts on $\mathcal{H}_\varphi$.
There exists a unique unitary operator $W\in \mathcal{B}(\mathcal{H}_\varphi\otimes\mathcal{H}_\varphi)$ which is known as multiplicative unitary defined by
$$W^*(\Lambda_\varphi(a)\otimes\Lambda_\varphi(b))=(\Lambda_\varphi\otimes\Lambda_\varphi)(\Delta(b)(a\otimes 1)), a,b\in\mathfrak{N}_\varphi.$$
Moreover for any $x\in\mathcal{M}$, $\Delta(x)=W^*(1\otimes x)W.$

For the locally compact quantum group $\mathbb{G}=(\mathcal{M},\Delta,\varphi,\psi)$ above, there exist the unitary antipode $R$, the scaling automorphism group $\tau_t$, $t\in\mathbb{R}$ and the antipode $S$ on $\mathcal{M}$. There exists a modular element $\delta$ such that $\psi=\varphi_\delta=\varphi R$. For the properties, we refer to \cite{KuVaes} for more details.

For $\mathbb{G}=(\mathcal{M},\Delta,\varphi,\psi)$, there always exist a dual locally compact quantum group $\hat{\mathbb{G}}=(\hat{\mathcal{M}},\hat{\Delta},\hat{\varphi},\hat{\psi})$. The corresponding von Neumann algebra acting on $\mathcal{H}_\varphi$ is given by
$$\hat{\mathcal{M}}=\{(\omega\otimes\iota)(W)|\omega\in \mathcal{B}(\mathcal{H}_\varphi)_*\}^{-\sigma-\text{strong}-*}.$$
The element $(\omega\otimes\iota)(W)$ is denoted by $\lambda(\omega)$ in general which is also know as the Fourier transform of the restriction $\omega|_\mathcal{M}$ of $\omega$ on $\mathcal{B}(\mathcal{H})$.
The comultiplication $\hat{\Delta}$ is given by
$$\hat{\Delta}(x)=\hat{W}^*(1\otimes x)\hat{W},\quad \hat{W}=\Sigma W^* \Sigma,$$
where $\Sigma$ is the flip on $\mathcal{H}_\varphi\otimes\mathcal{H}_\varphi$.
The dual left Haar weight $\hat{\varphi}$ is defined to be the unique normal, semi-finite, faithful weight on $\hat{\mathcal{M}}$ with GNS triple $(\mathcal{H}_\varphi,\iota,\hat{\Lambda})$ such that
$\lambda(\mathcal{I})$ is a core for $\hat{\Lambda}$ and $\hat{\Lambda}(\lambda(\omega))=\xi(\omega),\omega\in \mathcal{I}$, where
$$\mathcal{I}=\{\omega\in\mathcal{M}_*|\Lambda_\varphi(x)\mapsto\omega(x^*), x\in\mathfrak{N}_\varphi\text{ is bounded}\},$$
and $\xi(\omega)$ is given by $\omega(x^*)=\langle\xi(\omega),\Lambda_\varphi(x)\rangle$.
The dual right Haar weight $\hat{\psi}=\hat{\varphi}\hat{R}$, where $\hat{R}$ is the dual unitary antipode. For more details on dual quantum groups, we refer to \cite{KuVaes} again.

A locally compact quantum group $\mathbb{G}$ is a Kac algebra if its scaling automorphism group $\tau$ is trivial and $\sigma^\varphi=\sigma^\psi$. A locally compact quantum group is of compact type if $\varphi=\psi$ is a state.

Now we would like to recall some notations on noncommutative L$^p$ space.
Let $\mathcal{M}$ be a von Neumann algebra with a normal semifinite faithful weight $\varphi$. Denote by
$$\mathcal{T}_\varphi=\{x\in\mathcal{M}|x \text{ is analytic w.r.t. }\sigma \text{ and } \sigma_z(x)\in\mathfrak{N}_\varphi^*\cap\mathfrak{N}_\varphi,\forall z\in\mathbb{C}\}.$$
Let
\begin{eqnarray*}
\mathcal{L}_\varphi&=&\{x\in\mathcal{M}|\exists \varphi^{(x)}\in\mathcal{M}_*\text{ such that }\forall a,b\in\mathcal{T}_\varphi:\\
&&\varphi^{(x)}(a^*b)=\langle xJ_\varphi\nabla_\varphi^{-1/2}\Lambda_\varphi(a),J_\varphi\nabla_\varphi^{1/2}\Lambda_\varphi(b)\rangle\}
\end{eqnarray*}
Denote by $x\varphi$ the functional given by $x\varphi(y)=\varphi(yx)$ and  $\varphi x$ the functional given by $(\varphi x)(y)=\varphi(xy)$.
If there exists a bounded functional $\varphi^{(x)}\in\mathcal{M}_*$ such that $\varphi^{(x)}(y)=(x\varphi)(y)$ for any $y$ in the domain $\mathcal{D}(x\varphi)$,
then we denote $\varphi^{(x)}$ by $x\varphi$ again for simplicity i.e. $\varphi^{(x)}=x\varphi$ for $x$ in $\mathcal{L}_\varphi$.
We reserve $\varphi_x$ for $\varphi(x^{1/2}\cdot x^{1/2})$ when $x$ is a positive self-adjoint element affiliated with $\mathcal{M}$.
For any functional $\phi$, we denote by $\overline{\phi}$ the functional given by $\overline{\omega}(x)=\overline{\omega(x^*)}$ for any $x$ in $\mathcal{M}$.

Let $\mathcal{R}_\varphi=\{x^*|x\in\mathcal{L}_\varphi\}$. Then for any $x\in\mathcal{R}_\varphi$, $\varphi x\in\mathcal{M}_*$ under the convention above. In \cite{Cas}, Caspers showed that $\mathcal{T}_\varphi^2\subseteq \mathcal{L}_\varphi$.

For any $x\in\mathcal{L}_\varphi$, the norm is defined by
$$\|x\|_{\mathcal{L}_\varphi}=\max\{\|x\|,\|x\varphi\|\}.$$
Then for $p\in(1,\infty)$, $L^p(\mathbb{G})$ is the complex interpolation space $(\mathcal{M},\mathcal{M}_*)_{[1/p]}$ and $L^1(\mathbb{G})=\mathcal{M}_*$, $L^\infty(\mathbb{G})=\mathcal{M}$.

We quote the Theorem 4.1.2 in \cite{Bergh} for future use in the paper.
\begin{proposition}\label{map}
Let $\theta\in [0,1]$. Let $T$ be a morphism between compatible couples $(E_0,E_1)$ and $(F_0,F_1)$. Then it restricts a bounded linear map $T:(E_0,E_1)_{[\theta]}\to (F_0,F_1)_{[\theta]}$. The norm is bounded by $\|T\|\leq \|T:E_0\to F_0\|^{1-\theta}\|T: E_1\to F_1\|^{\theta}.$
\end{proposition}

In \cite{Cas}, Caspers proved that $\mathcal{H}_\varphi\cap\mathcal{M}=\mathfrak{N}_\varphi$ and $\mathcal{M}_*\cap\mathcal{H}_\varphi=\mathcal{I}$.
 Moveover $(\mathcal{H}_\varphi,\mathcal{M}_*)_{[2/p-1]}=L^p(\mathbb{G})$ for $p\in(1,2]$ and $(\mathcal{M},\mathcal{H}_\varphi)_{[2/q]}=L^q(\mathbb{G})$ for $q\in[2,\infty)$.
 For more details on this, we refer to \cite{Cas}.

Hence for $p\in[1,2]$ and $\omega\in\mathcal{I}$, the $L^p$-Fourier transform
$$\mathcal{F}_p:L^p(\mathbb{G})\to L^q(\widehat{\mathbb{G}}),\quad \frac{1}{p}+\frac{1}{q}=1,$$
is given by
$\mathcal{F}_p(\xi_p(\omega))=\hat{\Lambda}_q(\lambda(\omega))$,
where $\xi_p:\mathcal{I}\to L^p(\mathbb{G})$ is the inclusion map for $1\leq p\leq 2$, $\hat{\Lambda}_q:\mathfrak{N}_{\hat{\varphi}}\to L^q(\widehat{\mathbb{G}})$ is the inclusion map for $q\geq 2$.

The Hausdorff-Young inequality says that $\|\mathcal{F}_p\|\leq 1$.

Let $\phi$ be a normal, semi-finite faithful weight on $\mathcal{M}'$ acting on $\mathcal{H}_\varphi$. For $p\in[1,\infty)$, the Hilsum's $L^p$ space $L^p(\phi)$ is the space of closed densely defined operators $x$ on the GNS-space $\mathcal{H}_\varphi$ of $\varphi$ such that if $x=u|x|$ is the polar decomposition, then $|x|^p$ is the spatial derivative of a positive linear functional $\omega\in \mathcal{M}_*$ and $u\in\mathcal{M}$. For more details on noncommutative L$^p$ spaces, we refer \cite{Terp81}.
Let $d=\frac{d\varphi}{d\phi}$ be the spatial derivative relative to $\varphi$.
In \cite{Cas}, they prove that
there is an isometric isomorphism $\Phi_p: L^p(\phi)\to L^p(\mathbb{G})$ such that
$$\Phi_p: [abd^{1/p}]\mapsto l^p(ab), \quad a,b\in\mathcal{T}_\varphi,$$
where $l^p:\mathcal{L}_\varphi\to L^p(\mathbb{G})$ is the inclusion map, and $[x]$ is the closure of a preclosed operator $x$.

\begin{lemma}\label{normchange}
Let $\mathbb{G}=(\mathcal{M},\Delta,\varphi,\psi)$ be a locally compact quantum group. If $\alpha$ is an automorphism of $\mathcal{M}$, then for any $x$ in $\mathcal{L}_\varphi$,
$$\|x\|_{p,\varphi}=\|\alpha(x)\|_{p,\varphi\alpha^{-1}},$$
where p-norm $\|\cdot\|_{p,\eta}$ is the norm of complex interpolation L$^p$-space relative to a normal semi-finite weight $\eta$ on $\mathcal{M}$.

If $\alpha$ is an anti-automorphism of $\mathcal{M}$, then for any $x$ in $\mathcal{L}_\varphi$, we have
$$\|x\|_{p,\varphi}=\|\alpha(x^*)\|_{p,\varphi\alpha^{-1}}.$$
\end{lemma}
\begin{proof}
Directly from Proposition \ref{map}.
\end{proof}

\section{YOUNG'S INEQUALITY}
Let $\mathbb{G}=(\mathcal{M},\Delta,\varphi,\psi)$ be a locally compact quantum group. Suppose that $\omega,\theta\in \mathcal{M}_*$. Then the convolution $\omega *\theta\in\mathcal{M}_*$ of $\omega$ and $\theta$ is defined by
$$(\omega*\theta)(x)=(\omega\otimes\theta)\Delta(x)$$
for any $x$ in $\mathcal{M}$.
We then see that
\begin{equation}\label{eq1}
\|\omega*\theta\|\leq \|\omega\|\|\theta\|.
\end{equation}
If we identify $\mathcal{M}_*$ with $L^1(\mathbb{G})$, then the inequality (\ref{eq1}) is
\begin{equation}\label{eq11}
\|x*y\|_1\leq \|x\|_1\|y\|_1,\quad x,y\in L^1(\mathbb{G}).
\end{equation}

In \cite{KuVaes}, Kusterman and Vaes prove that for any $\omega\in\mathcal{M}_*$ and $\theta\in\mathcal{I}$,
$$\xi(\omega*\theta)=\lambda(\omega)\xi(\theta).$$
Note that $\lambda(\omega)=(\omega\otimes\iota)(W)$, we have
\begin{equation}\label{eq2}
\|\xi(\omega*\theta)\|\leq \|\omega\|\|\xi(\theta)\|.
\end{equation}
If we identify $\mathcal{H}_\varphi$ with $L^2(\mathbb{G})$, then the inequality (\ref{eq2}) is
\begin{equation}\label{eq21}
\|x*y\|_2\leq \|x\|_1\|y\|_2,\quad x\in L^1(\mathbb{G}), y\in L^1(\mathbb{G})\cap L^2(\mathbb{G}).
\end{equation}
Note that $L^1(\mathbb{G})\cap L^2(\mathbb{G})=\mathcal{I}$ is dense in $L^2(\mathbb{G})$.
Hence the convolution $x*y$ of $x\in L^1(\mathbb{G})$ and $y\in L^2(\mathbb{G})$ is well-defined by continuity and moreover
the inequality (\ref{eq21}) is true for $x\in L^1(\mathbb{G})$ and $y\in L^2(\mathbb{G})$.

Now by applying the interpolation theorem, we have
\begin{proposition}
For any $x\in L^1(\mathbb{G})$ and $y\in L^p(\mathbb{G})$, $1\leq p\leq 2$, we have
$$\|x*y\|_p\leq \|x\|_1\|y\|_p.$$
\end{proposition}

Recall that $\rho_t$ is the norm continuous one-parameter representation of $\mathbb{R}$ on $\mathcal{M}_*$ such that $\rho_t(\omega)=\omega(\delta^{-it}\tau_{-t}(x))$ for $\omega\in \mathcal{M}_*$, $x\in\mathcal{M}$ and $t\in\mathbb{R}$.
By \cite{KuVaes}, Remark 8.12,  we have that the set $\mathcal{I}_\rho=\{\omega\in\mathcal{I}|\omega\text{ is analytic with respect to } \rho\}$ is dense in $\mathcal{I}\subset \mathcal{M}_*$.
By \cite{Ku97}, Lemma 1.1, we see that $\xi(\mathcal{I}_\rho)$ is dense in $\xi(\mathcal{I})\subset\mathcal{H}_\varphi$.
Therefore $\xi_p(\mathcal{I}_\rho)$ is dense in $L^p(\mathbb{G})$ for $1\leq p\leq 2$.

In \cite{KuVaes}, Proposition 8.11, it was showed that for $\omega\in \mathcal{I}$ and $\theta\in\mathcal{D}(\rho_{i/2})$, $\omega*\theta\in\mathcal{I}$ and
$$\xi(\omega*\theta)=U^*\lambda(\rho_{i/2}(\theta))^*U\xi(\omega),$$
where $U:\mathcal{H}_\varphi\to\mathcal{H}_\varphi$ is an anti-unitary such that $U\Gamma(x)=\Lambda_\varphi(R(x^*))$ for any $x\in\mathfrak{N}_\psi$, where $(\mathcal{H}_\varphi,\pi,\Gamma)$ is the GNS-construction for $\psi=\varphi_\delta$ constructed from $(\mathcal{H}_\varphi,\pi,\Lambda_\varphi)$ via $\delta$ (See Notation 7.13 in \cite{KuVaes}).

Inspired by the equation above, we are able to show the following proposition.

\begin{proposition}
Suppose $1\leq p,q\leq 2$ and $\frac{1}{p}+\frac{1}{q}=\frac{3}{2}$.
Let $p',q'$ be such that $\frac{1}{p}+\frac{1}{p'}=1$, $\frac{1}{q}+\frac{1}{q'}=1$, Suppose that $\omega,\theta\in \mathcal{I}$ and $\theta$ is analytic with respect to $\rho$.
Then we have $\xi(\omega*\rho_{-i/p'}(\theta))\in L^2(\mathbb{G})$ and
$$\|\xi(\omega*\rho_{-i/p'}(\theta))\|_2\leq \|\xi_p(\omega)\|_p\|\xi_q(\theta)\|_q.$$
\end{proposition}
\begin{proof}
Note that $\omega*\rho_{-i/p'}(\theta)\in\mathcal{M}_*$. Since $\rho_{-i/p'}(\theta)\in \mathcal{I}$ (see Remark 8.12 in \cite{KuVaes}) and Result 8.6 in \cite{KuVaes}, we have $\omega*\rho_{-i/p'}(\theta)\in \mathcal{I}$. Hence
$$\xi(\omega*\rho_{-i/p'}(\theta))=\hat{\Lambda}(\lambda(\omega*\rho_{-i/p'}(\theta))).$$
Note that $\lambda(\omega*\rho_{-i/p'}(\theta))\in\mathfrak{N}_{\hat{\varphi}}$. By Theorem 23 in \cite{Terp82}, we have
$$\|\hat{\Lambda}(\lambda(\omega*\rho_{-i/p'}(\theta)))\|
=\|(\lambda(\omega*\rho_{-i/p'}(\theta))\hat{d}^{1/2}\|_{2,L^2(\hat{\phi})},$$
where $\hat{\phi}$ is a normal semi-finite faithful weight on $\widehat{M}'$ and $L^2(\hat{\phi})$ is the Hilsum space.

By the fact that $\lambda(\omega)\in\mathfrak{N}_{\hat{\varphi}}$ and the property of $\lambda$, we have
\begin{eqnarray*}
\|(\lambda(\omega*\rho_{-i/p'}(\theta))\hat{d}^{1/2}\|_{2,L^2(\hat{\phi})}
&=&\|\lambda(\omega)\lambda(\rho_{-i/p'}(\theta))\hat{d}^{1/p'}\hat{d}^{1/q'}\|_{2,L^2(\hat{\phi})}
\end{eqnarray*}
By Theorem 2.4 in \cite{Cooney} and Proposition 8.9 in \cite{KuVaes},
we have $\lambda(\rho_{-i/p'}(\theta))$ is analytic with respect to $\hat{\sigma}^{\hat{\varphi}}$ and
\begin{eqnarray*}
\|(\lambda(\omega*\rho_{-i/p'}(\theta))\hat{d}^{1/2}\|_{2,L^2(\hat{\phi})}&=&\|\lambda(\omega)\hat{d}^{1/p'}\lambda(\theta)\hat{d}^{1/q'}\|_{2,L^2(\hat{\phi})}\\
&\leq &\|\lambda(\omega)\hat{d}^{1/p'}\|_{p',L^{p'}(\hat{\phi})}\|\lambda(\theta)\hat{d}^{1/q'}\|_{q', L^{q'}(\hat{\phi})}.
\end{eqnarray*}
Now applying Hausdorff-Young inequality for locally compact quantum groups in \cite{Cooney}, we have
\begin{eqnarray*}
\|\xi(\omega*\rho_{-i/p'}(\theta))\|&\leq &\|\lambda(\omega)\hat{d}^{1/p'}\|_{p',L^{p'}(\hat{\phi})}\|\lambda(\theta)\hat{d}^{1/q'}\|_{q', L^{q'}(\hat{\phi})}\\
&\leq &\|\xi_p(\omega)\|_p\|\xi_q(\theta)\|_q.
\end{eqnarray*}
\end{proof}

Since $\xi_p(\mathcal{I})$ is dense in $L^p(\mathbb{G})$ for any $1\leq p\leq 2$, we then have
\begin{proposition}
Suppose $1\leq p,q\leq 2$ and $\frac{1}{p}+\frac{1}{q}=\frac{3}{2}$.
Let $p'$ be such that $\frac{1}{p}+\frac{1}{p'}=1$, $x\in L^p(\mathbb{G})$, $y\in L^q(\mathbb{G})$.
We define $x*\rho_{-i/p'}(y)$ to be the limit of $\xi(\omega_n*\rho_{-i/p'}(\theta_m))$ in $L^2(\mathbb{G})$, where $(\omega_n)_n\subset \mathcal{I}$ is a bounded net in $\mathcal{I}$ such that $(\xi_p(\omega_n))_n$ converges to $x$ in $L^p(\mathbb{G})$ and $(\theta_m)_m\subset \mathcal{I}$ is a bounded net such that $\theta_m$ is analytic with respect to $\rho$ and $(\xi_q(\theta_m))_m$ converges to $y$ in $L^q(\mathbb{G})$.
Then $x*\rho_{-i/p'}(y)\in L^2(\mathbb{G})$ and
$$\|x*\rho_{-i/p'}(y)\|_2\leq \|x\|_p\|y\|_q.$$
\end{proposition}
\begin{proof}
By the proposition above, we have
$$\|\xi(\omega_n*\rho_{-i/p'}(\theta_m))\|\leq \|\xi_p(\omega_n)\|_p\|\xi_q(\theta_m)\|_q.$$
By the assumption, we see that $\{\xi_p(\omega_n)\}_n$ and $\{\xi_q(\theta_m)\}_m$ are Cauchy nets, and hence $\{\xi(\omega_n*\rho_{-i/p'}(\theta_m))\}_{n,m}$ is a Cauchy net.
By taking the limits, we obtain that
$$\|x*\rho_{-i/p'}(y)\|_2\leq \|x\|_p\|y\|_q.$$
\end{proof}

Now by Stein's interpolation theorem \cite{Stein}, we have
\begin{theorem}\label{younglcqg}
Let $\mathbb{G}$ be a locally compact quantum group. For $1\leq p,q,r\leq 2$ with $\frac{1}{r}+1=\frac{1}{p}+\frac{1}{q}$, let $p'$ be such that $\frac{1}{p}+\frac{1}{p'}=1$.  Suppose $x\in L^p(\mathbb{G})$ and $y\in L^q(\mathbb{G})$. Then
$$\|x*\rho_{-i/p'}(y)\|_r\leq \|x\|_p\|y\|_q.$$
\end{theorem}

\begin{remark}
In Theorem \ref{younglcqg}, $x*\rho_{-i/p'}(y)$ is similarly defined to be the limit of $\xi_r(\omega_n*\rho_{-i/p'}(\theta_m))$ in $L^2(\mathbb{G})$, where $(\omega_n)_n\subset \mathcal{I}$ is a bounded net in $\mathcal{I}$ such that $(\xi_p(\omega_n))_n$ converges to $x$ in $L^p(\mathbb{G})$ and $(\theta_m)_m\subset \mathcal{I}$ is a bounded net such that $\theta_m$ is analytic with respect to $\rho$ and $(\xi_q(\theta_m))_m$ converges to $y$ in $L^q(\mathbb{G})$.
\end{remark}

\begin{corollary}
Let $\mathbb{G}$ be a compact quantum group. Then for $1\leq p,q,r,p'\leq 2$ with $\frac{1}{r}+1=\frac{1}{p}+\frac{1}{q}$, $\frac{1}{p}+\frac{1}{p'}=1$, $x\in L^p(\mathbb{G})$ and $y\in L^q(\mathbb{G})$, we have
$$\|x*\tau_{i/p'}(y)\|_r\leq \|x\|_p\|y\|_q.$$
\end{corollary}
\begin{proof}
Suppose that $x,y\in \mathcal{L}_\varphi=\mathcal{M}$. It suffices to compute $\rho_{-i/p'}(y\varphi)$.
We assume that $y\in\mathcal{D}(\tau_{i/p'})$ and $z\in\mathcal{D}(\tau_{-i/p'})$.
Then by Proposition 6.8 in \cite{KuVaes},
\begin{eqnarray*}
(\rho_{-i/p'}(y\varphi))(z)&=&(y\varphi)(\tau_{-i/p'}(z))\\
&=&\varphi\tau_{-i/p'}(z\tau_{i/p'}(y))=\varphi(z\tau_{i/p'}(y)).
\end{eqnarray*}
Since $\mathcal{D}(\tau_{-i/p'})$ is $\sigma$-strongly-* dense in $\mathcal{M}$, we have $\rho_{-i/p'}(y\varphi)=\tau_{i/p'}(y)\varphi$.
\end{proof}

\begin{proposition}
For $1\leq p,q,r\leq 2$, $x\in L^p(\mathbb{G})$ and $y\in L^q(\mathbb{G})$, we have
$$\hat{\Phi}_{r'}^{-1}\mathcal{F}_r(x*\rho_{-i/p'}(y))=\hat{\Phi}_{p'}^{-1}\mathcal{F}_p(x)\hat{\Phi}_{q'}^{-1}\mathcal{F}_q(y).$$
\end{proposition}
\begin{proof}
By continuity of $L^p$-Fourier transform, we only have to check
$$\hat{\Phi}_{r'}^{-1}\mathcal{F}_r(\xi_r(\omega*\rho_{-i/p'}(\theta)))=\hat{\Phi}_{p'}^{-1}\mathcal{F}_p(\xi_p(\omega))\hat{\Phi}_{q'}^{-1}\mathcal{F}_q(\xi_q(\theta))$$
for $\omega,\theta\in\mathcal{I}$ and $\theta$ analytic with respect to $\rho$.
Since $\omega*\rho_{-i/p'}(\theta)\in \mathcal{I}$, by Theorem 3.1 in \cite{Cooney}, we have
\begin{eqnarray*}
\hat{\Phi}_{r'}^{-1}\mathcal{F}_r(\xi_r(\omega*\rho_{-i/p'}(\theta)))&=&\lambda(\omega)\lambda(\rho_{-i/p'}(\theta))\hat{d}^{1/r'}\\
&=&\lambda(\omega)\hat{d}^{1/p'}\lambda(\theta)\hat{d}^{1/q'}\\
&=& \hat{\Phi}_{p'}^{-1}\mathcal{F}_p(\xi_p(\omega))\hat{\Phi}_{q'}^{-1}\mathcal{F}_q(\xi_q(\theta)),
\end{eqnarray*}
where the products are strong products.
\end{proof}

In general, we do not have Young's inequality for $1\leq p,q,r\leq \infty$. For example, let $\mathbb{G}=SU_\mu(2)$, we will show that $\|x*y\|_\infty\leq \|x\|_1\|y\|_\infty$ is not true for all $x\in L^1(\mathbb{G})$ and $y\in L^\infty(\mathbb{G})$.
Firstly we need the following proposition for convolutions.

\begin{proposition}
Let $\mathbb{G}$ be a compact quantum group.
Then the convolution $x*y$ of $x\in\mathcal{D}(S)$ and $y\in\mathcal{M}$ is given by
$$x*y=((x\varphi) S^{-1}\otimes \iota)\Delta(y).$$
The convolution $x*y$ of $x\in\mathcal{M}$ and $y\in\mathcal{D}(S^{-1})$ is given by
$$x*y=((\iota\otimes(y\varphi)S)\Delta(x)).$$
\end{proposition}
\begin{proof}
Since $\mathbb{G}$ is compact, $\mathcal{L}_\varphi=\mathcal{M}$.
If $x\in\mathcal{D}(S)$, $z\in\mathcal{D}(S^{-1})$
since $\|(x\varphi)S^{-1}(z)\|=\|\varphi S^{-1}(S(x)z)\|\leq \|S(x)\|\|z\|$,
we see that $(x\varphi)S^{-1}$ extends to a bounded linear functional on $\mathcal{M}$, denoted by $(x\varphi)S^{-1}$ again.
Then for any $z\in\mathcal{M}$,
\begin{eqnarray*}
((x\varphi)*(y\varphi))\Delta(z)&=&(x\varphi)((\iota\otimes\varphi)(\Delta(z)(1\otimes y)))\\
&=& (x\varphi)S^{-1}((\iota\otimes \varphi)(1\otimes z)\Delta(y))\\
&=&\varphi(z(((x\varphi)S^{-1}\otimes\iota)\Delta(y))).
\end{eqnarray*}
If $y\in \mathcal{D}(S^{-1})$, $z\in\mathcal{D}(S)$, we then have
$$\|(y\varphi) S(z)\|=\|\varphi S(zS^{-1}(y))\|\leq \|S^{-1}(y)\|\|z\|,$$
and $(y\varphi) S$ extends to a bounded linear functional on $\mathcal{M}$, denoted by $(y\varphi) S$ again.
Then for any $z\in\mathcal{M}$,
\begin{eqnarray*}
((x\varphi)*(y\varphi))\Delta(z)&=&(y\varphi)((\varphi\otimes\iota)(\Delta(z)(x\otimes 1)))\\
&=& (y\varphi)S((\varphi\otimes\iota)(z\otimes 1)\Delta(x))\\
&=&\varphi(z((\iota\otimes(y\varphi)S)\Delta(x))).
\end{eqnarray*}
\end{proof}

Now we consider $SU_\mu(2)$. For $\mu\in [-1,1]\backslash \{0\}$, $SU_\mu(2)$ is the universal unital C$^*$-algebra generated by $a,c$ subject to the conditions:
\begin{eqnarray*}
a^*a+c^*c=1,\quad aa^*+\mu^2c^*c=1,\\
c^*c=cc^*,\quad ac=\mu ca,\quad ac^*=\mu c^*a.
\end{eqnarray*}
Moreover $\|a^n\|=1$ for any $n\in\mathbb{N}$ and $(1-\mu^2)^{1/2}\leq\|c\|\leq 1$.

The comuliplication $\Delta$ on $SU_\mu(2)$ is given by
$$\Delta(a)=a\otimes a-\mu c^*\otimes c,\quad \Delta(c)=c\otimes a+a^*\otimes c.$$

The antipode $S$ on $SU_\mu(2)$ is given by
$$S(a)=a^*,\quad S(a^*)=a, S(c^*)=-\mu^{-1}c^*.$$

Let
\begin{eqnarray*}
a_{kmn}:=\left\{\begin{array}{ll}
a^k{c^*}^mc^n, & k\geq 0\\
{a^*}^{-k}{c^*}^mc^n, & k<0.
\end{array}\right.
\end{eqnarray*}

The Haar state $\varphi$ of $SU_\mu(2)$ is given by
$$\varphi(a_{kmn})=\delta_{k,0}\delta_{m,n}\frac{1-\mu^2}{1-\mu^{2m+2}}, \quad \mu\neq \pm 1.$$

Suppose $x={c^*}^{2n}$ and $y=c^{2n}$ for $n\in\mathbb{N}$. Then
$\Delta(y)=(c\otimes a+a^*\otimes c)^{2n}$ and
\begin{eqnarray*}
{c^*}^{2n}*c^{2n}&=&(({c^*}^{2n}\varphi)S^{-1}\otimes\iota)\Delta(c^{2n})\\
&=&(-\mu^{-1})^{2n}\varphi(c^{2n}{c^*}^{2n})a^{2n},
\end{eqnarray*}
and so $\|{c^*}^{2n}*c^{2n}\|=\mu^{-2n}\varphi({c^*}^{2n}*c^{2n})$.
On the other hand,
$\|{c^*}^{2n}\|_1=\varphi({c^*}^{n}c^{n})$
and
$(1-\mu^2)^{1/2}\leq \|c\|\leq 1$.
Hence
\begin{eqnarray*}
\frac{\|{c^*}^{2n}*c^{2n}\|}{\|{c^*}^{2n}\|_1\|\|c^{2n}\|}&\geq &\frac{\mu^{-2n}(1-\mu^2)(1-\mu^{2n+2})}{(1-\mu^2)(1-\mu^{4n+2})}=\frac{\mu^{-2n}(1-\mu^{2n+2})}{(1-\mu^{4n+2})}\\
&\geq &\mu^{-2n}(1-\mu^{2n+2}).
\end{eqnarray*}
Hence when $\mu\neq \pm1 $,
$$\sup_{0\neq x\in L^1(\mathbb{G}),0\neq y\in L^\infty(G)}\frac{\|x*y\|}{\|x\|_1\|y\|}=\infty.$$

When a locally compact quantum group $\mathbb{G}$ has trivial scaling automorphism group, we have Young's inequality for $1\leq p,q,r\leq \infty$ if $\varphi=\psi$ or $\varphi$ is tracial.

\begin{proposition}\label{Kac1}
Let $\mathbb{G}=(\mathcal{M},\Delta,\varphi,\psi)$ be a locally compact quantum group whose scaling automorphism group is trivial. Then the convolution $x*y$ of $x\in \mathcal{L}_\varphi$ and $y\in \mathcal{L}_\varphi$ is given by
$$x*y=((x\varphi) R\otimes \iota)\Delta(y)\in\mathcal{L}_\varphi.$$
\end{proposition}
\begin{proof}
Note that $\tau_t$ is trivial and $S=R$. By Proposition 1.22 in \cite{KuVaes}, for any $z\in\mathcal{R}_\varphi$,
\begin{eqnarray*}
((x\varphi)*(y\varphi))\Delta(z)&=&(x\varphi)((\iota\otimes\varphi)(\Delta(z)(1\otimes y)))\\
&=& (x\varphi)R((\iota\otimes \varphi)(1\otimes z)\Delta(y))\\
&=& ((x\varphi)R\otimes (\varphi z))\Delta(y))\\
&=&\varphi(z(((x\varphi)R\otimes\iota)\Delta(y))).
\end{eqnarray*}
Note that $\mathcal{R}_\varphi$ is $\sigma$-strongly-* dense in $\mathcal{M}$ and $\|x*y\|\leq \|x\varphi \|\|y\|<\infty$.
We have $x*y=((x\varphi) R\otimes \iota)\Delta(y)\in\mathcal{M}.$
Since $y\in\mathcal{L}_\varphi\subset\mathfrak{N}_\varphi$, we have $((x\varphi) R\otimes\iota)\Delta(y)\in\mathfrak{N}_\varphi$ by Result 2.3 in \cite{KuVaes}.
Therefore $x*y=((x\varphi) R\otimes \iota)\Delta(y)\in\mathcal{L}_\varphi$ by Proposition 2.14 in \cite{Cas}.
\end{proof}

\begin{corollary}\label{young1}
Let $\mathbb{G}$ be a locally compact quantum group whose scaling automorphism group is trivial. Then for $x\in L^1(\mathbb{G})$, $y\in L^p(\mathbb{G})$, $1\leq p\leq \infty$,
$$\|x*y\|_p\leq \|x\|_1\|y\|_p.$$
\end{corollary}
\begin{proof}
From Proposition \ref{Kac1}, we have $\|x*y\|\leq \|x\|_1\|y\|$. Recall $\|x*y\|_1\leq \|x\|_1\|y\|_1$ for $x,y\in L^1(\mathbb{G})$. Then by complex interpolation theorem we have
$$\|x*y\|_p\leq \|x\|_1\|y\|_p,$$
for all $x\in L^1(\mathbb{G})$ and $y\in L^p(\mathbb{G})$.
\end{proof}

Recall that $\delta^*_t(\omega)(x)=\omega(\delta^{it}x)$ for any $x\in\mathcal{M}$. By interchanging the role of $x,y$ in Proposition \ref{Kac1}, we have
\begin{proposition}
Let $\mathbb{G}$ be a locally compact quantum group whose scaling automorphism group is trivial. Suppose $\omega\in\mathcal{M}_*$ is analytic with respect to $\delta^*$. Then the convolution $x*\omega$ of $x\in\mathcal{L}_\varphi$ and $\omega$ is given by
$$x*\omega=(\iota\otimes \delta^*_{-i}(\omega)R)\Delta(x)\in\mathcal{M}.$$
If $y\in\mathcal{L}_\varphi$ and $y\varphi$ is analytic with respect to $\delta^*$, then
$$x*y=(\iota\otimes \delta^*_{-i}(y\varphi)R)\Delta(x)\in\mathcal{M}.$$
\end{proposition}

\begin{proof}
For any $z,x\in\mathcal{T}_\varphi^2$, $((x\varphi)*\omega)(z)=\varphi(((\iota\otimes\omega)\Delta(z))x).$

We define $e_n=\frac{n}{\sqrt{\pi}}\int \exp(-n^2t^2)\delta^{it}dt.$
Then $e_n$ is analytic with respect to $\sigma^\varphi$ which implies $\mathfrak{N}_\varphi e_n\subseteq \mathfrak{N}_\varphi$.

As in the proof of Proposition 8.11 in \cite{KuVaes}, we have
$$\delta^{-1/2}(\iota\otimes\omega)\Delta(e_nz)=(\iota\otimes\delta^*_{-i/2}(\omega))\Delta(\delta^{-1/2}e_nz)\in\mathfrak{N}_\psi^*.$$
Hence
\begin{eqnarray*}
\lefteqn{\varphi(((\iota\otimes\omega)\Delta(e_nz))xe_m)}\\
&=& \psi(\delta^{-1/2}((\iota\otimes\omega)\Delta(e_nz))xe_m\delta^{-1/2})\\
&=&\psi((\iota\otimes\delta^*_{-i/2}(\omega))\Delta(\delta^{-1/2}e_nz)xe_m\delta^{-1/2})\\
&=&\delta^*_{-i/2}(\omega)((\psi\otimes\iota)(\Delta(\delta^{-1/2}e_nz)(xe_m\delta^{-1/2}\otimes 1)))\\
&=&\delta^*_{-i/2}(\omega)R((\psi\otimes\iota)((\delta^{-1/2}e_nz\otimes 1)\Delta(xe_m\delta^{-1/2})))\\
&=&\delta^*_{-i/2}(\omega)R((\varphi\otimes\iota)((e_nz\otimes 1)\Delta(xe_m\delta^{-1/2})(\delta^{1/2}\otimes 1))).
\end{eqnarray*}
Now applying $\Delta(\delta)=\delta\otimes\delta$, we obtain
\begin{eqnarray*}
\lefteqn{\varphi(((\iota\otimes\omega)\Delta(e_nz))xe_m)}\\
&=&\delta^*_{-i/2}(\omega)R((\varphi\otimes\iota)((e_nz\otimes 1)\Delta(xe_m))\delta^{-1/2})\\
&=&\delta^*_{-i/2}(\omega)(\delta^{1/2}R((\varphi\otimes\iota)((e_nz\otimes 1)\Delta(xe_m)))\\
&=&\delta^*_{-i}(\omega)(R((\varphi\otimes\iota)((e_nz\otimes 1)\Delta(xe_m)))\\
&=&\varphi(e_nz (\iota\otimes\delta^*_{-i}(\omega) R)\Delta(xe_m)).
\end{eqnarray*}
Hence $x*\omega=(\iota\otimes\delta^*_{-i}(\omega)R)\Delta(x)$ for $x\in\mathcal{T}_\varphi^2$. Since $\mathcal{T}_\varphi^2$ is $\sigma$-strongly-* dense in $\mathcal{L}_\varphi$, we have
$x*\omega=(\iota\otimes\delta^*_{-i}(\omega)R)\Delta(x)$ for all $x\in\mathcal{L}_\varphi$.
\end{proof}

\begin{proposition}
Let $\mathbb{G}$ be a locally compact quantum group whose scaling automorphism group is trivial. Suppose $1\leq p,q\leq \infty $, $\frac{1}{p}+\frac{1}{q}=1$.
If $\varphi=\psi$, then for $x\in L^p(\mathbb{G})$ and $y\in L^q(\mathbb{G})$, we have
$$\|x*y\|_\infty \leq \|x\|_p\|y\|_q.$$
If  $\varphi$ is tracial, then for $x\in L^p(\mathbb{G})$ and $y\in L^q(\mathbb{G})$, we have
$$\|x*y\delta^{-1/q}\|_\infty \leq \|x\|_p\|y\|_q.$$
\end{proposition}
\begin{proof}
Suppose that $\varphi=\psi$. Note that for any $x,y\in\mathcal{L}_\varphi$,
$$\|x*y\|_\infty=\sup_{\|z^*\|_1=1, z\in\mathcal{R}_\varphi}\varphi(z(x*y)).$$
Now we will calculate $\varphi(z(x*y))$.
By the condition $\varphi=\psi$, we see that $\varphi z=(R(z)\varphi)R$ and then
\begin{eqnarray*}
\varphi(z(x*y))&=&(x\varphi\otimes y\varphi)\Delta(z)\\
&=&(x\varphi)((\iota\otimes\varphi)(\Delta(z)(1\otimes y)))\\
&=&(x\varphi R)((\iota\otimes \varphi)(1\otimes z)\Delta(y))\\
&=&(x\varphi R\otimes \varphi z)\Delta(y)\\
&=&(x\varphi R\otimes R(z)\varphi R)\Delta(y)\\
&=&(R(z)\varphi\otimes x\varphi)\Delta(R(y))\\
&=&\varphi(R(y)(R(z)*x)).
\end{eqnarray*}
Moreover we assume that $x,y,z\in\mathcal{T}_\varphi^2$.
Let $\phi$ be a normal semi-finite faithful weight on $\mathcal{M}'$.
By Corollary \ref{young1} and Lemma \ref{normchange}, we have
\begin{eqnarray*}
|\varphi(z(x*y))|&=&|\varphi(R(y)(R(z)*x))|\\
&=&|\int (R(z)*x)d R(y)d\phi|\\
&\leq &\|R(z)*x\|_p\|R(y^*)\|_q\\
&\leq&\|R(z)\|_1\|x\|_p\|y\|_q\\
&=&\|z^*\|_1\|x\|_p\|y\|_q.
\end{eqnarray*}
Therefore $\|x*y\|_\infty\leq \|x\|_p\|y\|_q$.

Suppose that $\varphi$ is tracial. Let $e_n=\frac{n}{\sqrt{\pi}}\int\exp(-n^2t^2)\delta^{it}dt$.
Then we have
\begin{eqnarray*}
|\varphi(e_mz(x*ye_n\delta^{-1/q}))|&=&|(x\varphi R\otimes \varphi e_mz)(\Delta(ye_n\delta^{-1/q}))|\\
&=&|(x\varphi R\otimes R(z)e_m\delta\varphi R)(\Delta(ye_n\delta^{-1/q}))|\\
&=&|(R(z)e_m\delta\varphi\otimes x\varphi)\Delta(\delta^{1/q}e_nR(y))|\\
&=&|\varphi(\delta^{1/q}e_nR(y)(R(z)e_m\delta*x))|\\
&\leq &\|\delta^{1/q}e_nR(y)\|_q\|R(z)e_m\delta*x\|_p\\
&\leq &\|ye_n\|_q\|z^*e_m\|_1\|x\|_p.
\end{eqnarray*}
The last inequality follows from Corollary \ref{young1}.
Hence
$$\|x*y\delta^{-1/q}\|_\infty\leq \|x\|_p\|y\|_q.$$
\end{proof}

\begin{theorem}\label{youngkac}
Let $\mathbb{G}$ be a locally compact quantum group whose scaling automorphism group is trivial. Suppose $1\leq p,q,r\leq \infty $, $\frac{1}{p}+\frac{1}{q}=\frac{1}{r}$.
If $\varphi=\psi$, then for $x\in L^p(\mathbb{G})$ and $y\in L^q(\mathbb{G})$, we have
$$\|x*y\|_r \leq \|x\|_p\|y\|_q.$$
If $\varphi$ is tracial, then for $x\in L^p(\mathbb{G})$ and $y\in L^q(\mathbb{G})$, we have
$$\|x*y\delta^{-1/p'}\|_r \leq \|x\|_p\|y\|_q.$$
\end{theorem}
\begin{proof}
Directly from the interpolation theorem.
\end{proof}

\section{SHIFTS OF GROUP-LIKE PROJECTIONS}
Suppose that $\mathbb{G}=(\mathcal{M},\Delta,\varphi,\psi)$ is a locally compact quantum group.
A projection $h$ in $L^\infty(\mathbb{G})$ is a group-like projection if $\Delta(h)(1\otimes h)=h\otimes h$ and $h\neq 0$.

A projection $h$ in $L^\infty(\mathbb{G})\cap L^1(\mathbb{G})$ is a biprojection if $\mathcal{F}_1(h\varphi)$ is a multiple of a projection in $L^\infty(\hat{\mathbb{G}})$.

\begin{remark}
In \cite{LanDaele, Daele07}, the group-like projection is defined for *-algebraic quantum group.
In subfactor theory, the biprojection is defined for planar algebras etc.
\end{remark}

\begin{proposition}\label{glp}
Suppose $\mathbb{G}$ is a locally compact quantum group and $h\in \mathfrak{N}_\varphi$ is a group-like projection.  Then
\begin{itemize}
\item[(1)] $S(h)=h$, $R(h)=h$, and $\tau_t(h)=h$ for all $t\in\mathbb{R}$. Moreover the scaling constant $\nu=1.$
\item[(2)] $\Delta(h)(h\otimes 1)=h\otimes h.$
\item[(3)] $h=\sigma_t^\varphi(h)=\sigma_t^\psi(h)$ for all $t\in\mathbb{R}$.
\item[(4)] $h\varphi=h\psi$.
\end{itemize}
\end{proposition}
\begin{proof}
$(1)$ Note that $\varphi(h)=\varphi(h^*h)<\infty$, we have $h\in \mathfrak{N}_\varphi\cap\mathfrak{N}_\varphi^*$ and $R(h)\in\mathfrak{N}_\psi^*\cap\mathfrak{N}_\psi$. Applying $\iota\otimes\varphi$ to $\Delta(h)(1\otimes h)=h\otimes h$, we obtain $(\iota\otimes \varphi)(\Delta(h)(1\otimes h))=\varphi(h)h$.
Then $h$ is in $\mathcal{D}(S)$ and
\begin{eqnarray*}
\varphi(h)S(h)&=&S((\iota\otimes\varphi)(\Delta(h)(1\otimes h)))\\
&=&(\iota\otimes\varphi)((1\otimes h)\Delta(h))\\
&=&(\iota\otimes\varphi)(\Delta(h)(1\otimes h))=\varphi(h)h,
\end{eqnarray*}
i.e. $S(h)=h$.

By Proposition 5.5 in \cite{KuVaes}, we have that
$$(\psi\otimes\iota)(\Delta(R(h))(R(h)\otimes 1))G\subseteq G(\psi\otimes\iota)(\Delta(R(h))(R(h)\otimes 1)).$$
Applying the equation $\chi(R\otimes R)\Delta=\Delta R$, we see that
\begin{eqnarray*}
(\psi\otimes\iota)(\Delta(R(h))(R(h)\otimes 1))&=&(R\otimes \psi R)((1\otimes h)\Delta(h))\\
&=&(\varphi\otimes R)(h\otimes h)=\varphi(h)R(h).
\end{eqnarray*}
Hence $R(h)G\subseteq GR(h)$ and $R(h)G^*\subseteq G^*R(h)$. Since $G=IN^{1/2}$, we obtain that $R(h)N\subseteq NR(h)$ and $R(h)N^{it}=N^{it}R(h)$.
Since $\tau_t(x)=N^{-it}xN^{it}$, we see that $R\tau_t(h)=R(h)$ and $\tau_t(h)=h$. Hence $h$ is analytic with respect to $\tau$ and $\tau_{\pm i/2}(h)=h$. Finally $R(h)=h$.

There is another way to show $R(h)=h$ and $\tau_t(h)=h$.
By $S(h)=h$, we have $\tau_{-i}(h)=h$. Let $h_n=\frac{n}{\sqrt{\pi}}\int \exp(-n^2t^2)\tau_t(h)dt$.
Then $\tau_{-i}(h_n)=h_n$. This implies that $\tau_t(h_n)=h_n$ and $\tau_t(h)=h$.
Hence $h$ is analytic with respect to $\tau$ and $\tau_{\pm i/2}(h)=h$. Then we can obtain $R(h)=h$.

By Proposition 6.8 in \cite{KuVaes}, we have $\varphi(h)=\varphi(\tau_t(h))=\nu^{-t}\varphi(h)$ and $\nu^{-t}=1$ for any $t\in\mathbb{R}$.
This implies that $\nu=1$.

$(2)$
\begin{eqnarray*}
\Delta(h)(h\otimes 1)&=&\Delta(R(h))(R(h)\otimes 1)\\
&=& ((R\otimes R)\chi\Delta(h))(R(h)\otimes 1)\\
&=& \chi((R\otimes R)(\Delta(h))(1\otimes R(h)))\\
&=&\chi (R\otimes R)((1\otimes h)\Delta(h))\\
&=&\chi (R\otimes R) (h\otimes h)\\
&=&R(h)\otimes R(h)=h\otimes h
\end{eqnarray*}

$(3)$ By the relation $\Delta\tau_t=(\sigma_t^\varphi\otimes\sigma^\psi_{-t})\Delta$ in Proposition 6.8 in \cite{KuVaes}, we have for any $n\in\mathbb{N}$,
\begin{eqnarray*}
\Delta(h)&=&\Delta(\tau_t(h))\\
&=&\frac{n}{\sqrt{\pi}}\int\exp(-n^2t^2)\Delta(\tau_t(h))dt\\
&=&\frac{n}{\sqrt{\pi}}\int\exp(-n^2t^2)(\sigma_t^\varphi\otimes\sigma_{-t}^\psi)\Delta(h)dt,
\end{eqnarray*}
i.e. $\Delta(h)$ is analytic with respect to $\sigma^\varphi\otimes{(\sigma^\psi)}^{-1}$. Moreover $(\sigma_i^\varphi\otimes\sigma_{-i}^\psi)\Delta(h)=\Delta(h).$

Let $h_n=\frac{n}{\sqrt{\pi}}\int \exp(-n^2t^2)\sigma_t^\psi(h)dt$. Then $h_n$ is analytic with respect to $\sigma^\psi$.
Since $R(h)=h$ and $R\sigma_t^\psi=\sigma_{-t}^\varphi R$, we have $R(h_n)=\frac{n}{\sqrt{\pi}}\int \exp(-n^2t^2)\sigma_{-t}^\varphi(h)dt$, i.e. $R(h_n)$ is analytic with respect to $\sigma^\varphi$.

By Kaplansky density theorem, there is a net $\{f^\varphi_k\}_k$ of self-adjoint elements in the unit ball of $\mathfrak{N}_\varphi\cap\mathfrak{N}_\varphi^*$ such that $f_k^\varphi\to 1$ in the $\sigma$-strong-* topology.
We then define the net of elements $\{e_k^\varphi\}$ by
$$e_k^\varphi=\frac{1}{\sqrt{\pi}}\int \exp(-t^2)\sigma_t^\varphi(f_k^\varphi)dt.$$
It is known that $\sigma_z^\varphi(e_k^\varphi)\to 1$ in the $\sigma$-strong-* topology for $z\in\mathbb{C}$ and $\|\sigma_z^\varphi(e_k^\varphi)\|\leq \exp((\Im z)^2)$, where $\Im z$ is the image part of $z$.

Now
\begin{equation*}
\begin{aligned}
&(\iota\otimes\psi)((1\otimes e_k^\psi)\Delta(h)(1\otimes h_n))\\
&=(\iota\otimes\psi)\frac{n}{\sqrt{\pi}}\int \exp(-n^2t^2)(1\otimes e_k^\psi)\Delta(h)(1\otimes \sigma_{t}^\psi(h))dt\\
&=\frac{n}{\sqrt{\pi}}\int \exp(-n^2t^2)(\sigma_{-t}^\varphi\otimes\psi)(1\otimes \sigma_{-t}^\psi(e_k^\psi))(\sigma_t^\varphi\otimes\sigma_{-t}^\psi)(\Delta(h))(1\otimes h)dt\\
&=\frac{n}{\sqrt{\pi}}\int \exp(-n^2t^2)(\sigma_{-t}^\varphi\otimes\psi)(1\otimes \sigma_{-t}^\psi(e_k^\psi))(\Delta(\tau_t(h))(1\otimes h))dt\\
&=\frac{n}{\sqrt{\pi}}\int \exp(-n^2t^2)(\sigma_{-t}^\varphi(h)\otimes\psi(\sigma_{-t}^\psi(e_k^\psi)h)dt.
\end{aligned}
\end{equation*}
Since $e_k^\psi\to 1$ in $\sigma$-strong-* topology, we have
$$(\iota\otimes h_n\psi)(\Delta(h))=(\iota\otimes \psi h_n)(\Delta(h)).$$
Repeating the above calculation with $h_n$ and $e_k^\psi$ switched, we obtain that
$$(\iota\otimes h_n\psi)(\Delta(h))=R(h_n)\psi(h).$$

Then for any $a\in\mathcal{T}_\varphi$, we obtain
\begin{eqnarray*}
\lefteqn{(\varphi\otimes\psi)((a\otimes h_ne_{k}^\psi)(\sigma_i^{\varphi}\otimes\sigma_{-i}^\psi)(\Delta(h)(e_m^\varphi\otimes e_k^\psi)))}\\
&=&(\varphi\otimes\psi)((\sigma_{-i}^\varphi(a)\otimes \sigma^\psi_{i}(h_ne_{k}^\psi))\Delta(h)(e_m^\varphi\otimes e_k^\psi))\\
&=&(e_m^\varphi\varphi \sigma_{-i}^\varphi(a)\otimes\psi)((1\otimes \sigma^\psi_i(e_{k}^\psi) )\Delta(h)(1\otimes e_k^\psi h_n))\\
&=&((e_m^\varphi\varphi \sigma_{-i}^\varphi(a))\otimes (e_k^\psi h_n\psi \sigma_i^\psi(e_k^\psi)))(\Delta(h)).
\end{eqnarray*}
Let $e_m^\varphi\to 1$ and $e_k^\psi\to 1$ in $\sigma$-strong-* topology. We have

\begin{eqnarray*}
(\varphi\otimes \psi h_n)((a\otimes 1)\Delta(h))&=&(\varphi\otimes \psi h_n)((a\otimes 1)(\sigma^\varphi_i\otimes \sigma_{-i}^\psi)\Delta(h))\\
&=&(\varphi\otimes h_n\psi)((\sigma_{-i}^\varphi(a)\otimes 1)\Delta(h))\\
&=&\varphi(\sigma_{-i}^\varphi(a)R(h_n))\psi(h)
\end{eqnarray*}
and on the other hand,
\begin{eqnarray*}
(\varphi\otimes \psi h_n)((a\otimes 1)\Delta(h))=\varphi(aR(h_n))\psi(h).
\end{eqnarray*}
Now we see that $\varphi(a\sigma_i^\varphi(R(h_n)))=\varphi(aR(h_n))$ for any $a\in\mathcal{T}_\varphi$, and hence we obtain that $\sigma_i^\varphi (R(h_n)) =R(h_n)$.
Furthermore $\sigma_t^\varphi(R(h_n))=R(h_n)$ and $\sigma_t^\varphi(h)=h$ for $t\in\mathbb{R}$.
Applying $R\sigma_t^\psi=\sigma^\varphi_{-t} R$, we have $\sigma_t^\psi(h)=h$ for $t\in\mathbb{R}$.

$(4)$ Since $R(h)=h$, $h\in\mathfrak{N}_\psi\cap\mathfrak{N}_\varphi$. For any $a$ in $\mathfrak{N}_\psi^*\cap\mathfrak{N}_\varphi^*$, we have
\begin{eqnarray*}
\psi(ah)\varphi(h)&=&(\psi\otimes h\varphi)\Delta(ah)\\
&=&(\psi\otimes\varphi)(\Delta(a)\Delta(h)(1\otimes h))\\
&=&(\psi\otimes\varphi)(\Delta(a)\Delta(h)(h\otimes 1))\\
&=&(h\psi\otimes\varphi)\Delta(ah)\\
&=&\psi(h)\varphi(ah)=\varphi(h)\varphi(ah),
\end{eqnarray*}
i.e. $\psi(ah)=\varphi(ah)$.
By Proposition 1.14 in \cite{KuVaes99} and Proposition 6.8 in \cite{KuVaes}, $\mathfrak{N}_\varphi\cap\mathfrak{N}_\psi$ is a core for $\Lambda_\varphi$ and hence $h\psi=h\varphi$.
\end{proof}

\begin{remark}
Proposition \ref{glp} indicate that a locally compact quantum group $\mathbb{G}$ whose scaling constant $\nu\neq 1$ has no group-like projection in $L^1(\mathbb{G})\cap L^\infty(\mathbb{G})$.
A compact quantum group always has a group-like projection in $L^1(\mathbb{G})\cap L^\infty(\mathbb{G})$ as the identity is such a projection.
\end{remark}

\begin{proposition}\label{glpbi}
Suppose $\mathbb{G}$ is a locally compact quantum group and $h$ is a group-like projection in $\mathcal{L}_\varphi$, then $\varphi(h)^{-1}\mathcal{F}_1(h)$ is a group-like projection.
Moreover, $\varphi(h)\hat{\varphi}(\mathcal{R}(\mathcal{F}_1(h)))=1$, where $\mathcal{R}(\mathcal{F}_1(h))$ is the range projection of $\mathcal{F}_1(h)$.
\end{proposition}
\begin{proof}
Since there is a group-like projection $h$, by Proposition \ref{glp}, we see that $\nu=1$.
For any $b$ in $\mathcal{D}(S^{-1})$, by Proposition 6.8 in \cite{KuVaes} and Proposition \ref{glp},
we have that
$$|\varphi(S^{-1}(b)h)|=|\varphi(S^{-1}(hb))|=|\varphi (R(hb))|=|h\varphi R(b)|\leq \|h\varphi\|\|b\|$$
which implies that $h\varphi S^{-1}$ extends a bounded linear functional $h\varphi R$ on $\mathcal{M}$.

For any $a$ in $\mathcal{R}_\varphi$, we have
\begin{eqnarray*}
(h\varphi*h\varphi)(a)&=&(h\varphi \otimes h\varphi)\Delta(a)\\
&=&h\varphi S^{-1}((\iota\otimes\varphi)((1\otimes a)\Delta(h)))\\
&=&(h\varphi S^{-1}\otimes \varphi a)\Delta(h)\\
&=&\varphi a(( h\varphi R\otimes \iota)\Delta(h))\\
&=&\varphi a R((\iota\otimes h\varphi)\Delta(h))\\
&=&\varphi a R(\varphi(h)h)=\varphi(ah)\varphi(h),
\end{eqnarray*}
i.e. $h*h=\varphi(h)h$.
By taking the Fourier transform $\mathcal{F}_1$, we obtain $\lambda(h\varphi)^2=\varphi(h)\lambda(h\varphi)$.
For any $b$ in $\mathcal{D}(S)$, we have that
$$|\overline{h\varphi}(S(b))|=|\varphi(S(hb))|=|h\varphi R(b)|\leq \|h\varphi\|\|b\|.$$
This implies that $\overline{h\varphi}S$ extends to a bounded linear functional on $\mathcal{M}$. Hence $(h\varphi)^*\in\mathcal{M}_*$. By Proposition 2.4 in \cite{KuVaes03} and Proposition \ref{glp}, we have that
\begin{eqnarray*}
(\lambda(h\varphi))^*&=&\lambda((h\varphi)^*)=\lambda(\overline{h\varphi} S)\\
&=&\lambda(\overline{h\varphi} R)=\lambda(h\psi)=\lambda(h\varphi).
\end{eqnarray*}
Therefore $\varphi(h)^{-1}\lambda(h\varphi)$ is a projection in $L^\infty(\hat{\mathbb{G}})$.

Now by a routine computation, we obtain
\begin{eqnarray*}
\varphi(h)\hat{\varphi}(\mathcal{R}(\mathcal{F}_1(h)))&=&\hat{\varphi}(\mathcal{F}_1(h))\\
&=&\varphi(h)^{-1}\hat{\varphi}(\lambda(h\varphi)^*\lambda(h\varphi))\\
&=&\varphi(h)^{-1}\varphi(h^*h)=1.
\end{eqnarray*}

To see $\varphi(h)^{-1}\mathcal{F}_1(h)$ is a group-like projection, we have to check
$$\hat{\Delta}(\lambda(h\varphi))(1\otimes \lambda(h\varphi))=\lambda(h\varphi)\otimes \lambda(h\varphi).$$
Applying $\hat{\Lambda}\otimes\hat{\Lambda}$, we se that
\begin{eqnarray*}
(\hat{\Lambda}\otimes\hat{\Lambda})(\hat{\Delta}(\lambda(h\varphi))(\lambda(h\varphi)\otimes 1))&=&\hat{W}^*(\hat{\Lambda}(\lambda(h\varphi))\otimes \hat{\Lambda}(\lambda(h\varphi)) )\\
&=&\Sigma W\Sigma(\Lambda(h)\otimes \Lambda(h))\\
&=&\Sigma W ((\Lambda\otimes\Lambda)(h\otimes h))\\
&=& \Sigma W(\Lambda\otimes\Lambda)(\Delta(h)(1\otimes h))\\
&=&\Sigma ((\Lambda\otimes\Lambda)(h\otimes h))\\
&=&(\Lambda\otimes\Lambda)(h\otimes h)\\
&=&\hat{\Lambda}(\lambda(h\varphi))\otimes \hat{\Lambda}(\lambda(h\varphi)).
\end{eqnarray*}
Hence $\hat{\Delta}(\lambda(h\varphi))(\lambda(h\varphi)\otimes 1)=\lambda(h\varphi)\otimes \lambda(h\varphi)$.
By the equation $\hat{R}(\lambda(\omega))=\lambda(\omega R)$ in Proposition 8.17 of \cite{KuVaes} and $h\varphi R=h\varphi$, we have that
$$\hat{\Delta}(\lambda(h\varphi))(1\otimes \lambda(h\varphi))=\lambda(h\varphi)\otimes \lambda(h\varphi).$$
Hence $\varphi(h)^{-1}\mathcal{F}_1(h)$ is a group-like projection in $L^\infty(\hat{\mathbb{G}})$.
\end{proof}

\begin{remark}\label{glpbir}
In Proposition \ref{glpbi}, we see that $\mathcal{R}(\mathcal{F}_1(h))=\varphi(h)^{-1}\mathcal{F}_1(h)$ is a group-like projection in $\mathcal{L}_{\hat{\varphi}}$.
Now we can apply the Fourier transform $\hat{\mathcal{F}}_1$ to obtain that $\varphi(h)^{-1}h=\hat{\mathcal{F}}_1(\mathcal{R}(\mathcal{F}_1(h)))$.
\end{remark}

\begin{corollary}\label{mo}
Let $\mathbb{G}$ be a locally compact quantum group. Suppose that $h$ is a group-like projection in $\mathcal{L}_\varphi$.
Then $\delta^{it} h=h\delta^{it}=h$ for all $t\in\mathbb{R}$.
\end{corollary}
\begin{proof}
If $h$ is a group-like projection in $\mathcal{L}_\varphi$. Then $\lambda(h\varphi)$ is a group-like projection again.
By Proposition \ref{glp}, we have that $\hat{\sigma}_t^{\hat{\varphi}}(\lambda(h\varphi))=\lambda(h\varphi)$.
Hence $\rho_t(h\varphi)=h\varphi$ for any $t\in\mathbb{R}$. Now for any $b$ in $\mathfrak{N}_\varphi^*$, we have
\begin{eqnarray*}
\varphi(bh)=(h\varphi)(b)=\rho_t(h\varphi)(b)=\varphi(\delta^{-it}\tau_{-t}(b)h)
=\varphi(\delta^{-it}b h)=\varphi(bh\delta^{-it}).
\end{eqnarray*}
This implies that $h=h\delta^{-it}$ for all $t\in\mathbb{R}$.
\end{proof}

\begin{proposition}
Suppose $\mathbb{G}$ a locally compact quantum  group. Suppose that $\varphi$ is tracial or $\varphi=\psi$. Then a projection $h\in \mathcal{L}_\varphi$ is a group-like projection if and only if $h$ is a biprojection.
\end{proposition}
\begin{proof}
Proposition \ref{glpbi} showed that if $h\in \mathcal{L}_\varphi$ is a group-like projection, then $h$ is a biprojection. Now we will prove the reverse.

Suppose $h$ is a biprojection. Then $\mathcal{F}_1(h)$ is a multiple of a projection. Suppose $\mathcal{F}_1(h)=\lambda(h\varphi)=\mu h_0$ for some projection $h_0$ in $L^\infty(\hat{\mathbb{G}})$ and $\mu\in\mathbb{C}\backslash\{0\}$. Then $\lambda(h\varphi)^2=\mu \lambda(h\varphi)$ and $\lambda(h\varphi)^*=\frac{\bar{\mu}}{\mu}\lambda(h\varphi)$.

By Proposition 2.4 in \cite{KuVaes03}, we have that $(h\varphi)^*$ is bounded i.e. $\overline{h\varphi}S$ extends to a bounded linear functional on $\mathcal{M}$.
By $(h\varphi)^*=\frac{\bar{\mu}}{\mu}h\varphi$,
we see that $h\varphi S^{-1}$ extends to bounded linear functional $\frac{\mu}{\bar{\mu}}\overline{(h\varphi)}$ on $\mathcal{M}$. Hence $h\varphi S^{-1}=\frac{\mu}{\bar{\mu}}\overline{(h\varphi)}=\frac{\mu}{\bar{\mu}}\varphi h$.

For any $a$ in $\mathcal{R}_\varphi$, we have
\begin{eqnarray*}
(h\varphi*h\varphi)(a)&=&h\varphi((\iota\otimes h\varphi)\Delta(a))\\
&=&h\varphi((\iota\otimes\varphi)(\Delta(a)(1\otimes h)))\\
&=&h\varphi S^{-1}((\iota\otimes\varphi)(1\otimes a)\Delta(h))\\
&=&\frac{\mu}{\bar{\mu}}(\varphi h\otimes \varphi a)\Delta(h)\\
&=&\frac{\mu}{\bar{\mu}}\varphi(a(\varphi h\otimes\iota)(\Delta(h))
\end{eqnarray*}
i.e. $\bar{\mu} h=(\varphi h\otimes\iota)(\Delta(h)).$
Applying $\varphi$ to the equation, we have
$$\bar{\mu}\varphi(h)=(\varphi h\otimes \varphi)(\Delta(h))=\varphi(h)^2.$$
Hence $\mu=\varphi(h)$.

 Note that
\begin{eqnarray*}
\varphi(h)h&=&(\varphi h\otimes\iota)(\Delta(h))\\
&=&(h\varphi\otimes\iota)(\Delta(h))\\
&=&(h\varphi\otimes\iota)(\Delta(h)(1\otimes h))\\
&=&(\varphi h\otimes\iota)((1\otimes h)\Delta(h))\\
&=&(h\varphi \otimes\iota)((1\otimes h)\Delta(h)(1\otimes h))\\
&=&(\varphi h\otimes\iota)((1\otimes h)\Delta(h)(1\otimes h)),
\end{eqnarray*}
and
\begin{eqnarray*}
\lefteqn{((1\otimes h)\Delta(h)-h\otimes h)(\Delta(h)(1\otimes h)-h\otimes h)}\\
&=& (1\otimes h)\Delta(h)(1\otimes h)-(1\otimes h)\Delta(h) (h\otimes h)-(h\otimes h)\Delta(h)(1\otimes h)+h\otimes h.
\end{eqnarray*}

If $\varphi$ is tracial, we have
\begin{eqnarray*}
\lefteqn{(\varphi\otimes\varphi)((1\otimes h)\Delta(h)-h\otimes h)(\Delta(h)(1\otimes h)-h\otimes h))}\\
&=& (\varphi\otimes\varphi)((1\otimes h)\Delta(h)(1\otimes h))-(\varphi\otimes\varphi)((1\otimes h)\Delta(h) (h\otimes h))\\
&&-(\varphi\otimes\varphi)((h\otimes h)\Delta(h)(1\otimes h))+(\varphi\otimes\varphi)(h\otimes h)\\
&=&\varphi(h)^2-\varphi(h)^2-\varphi(h)^2+\varphi(h)^2=0
\end{eqnarray*}
i.e. $(1\otimes h)\Delta(h)=h\otimes h$. This indicates that $h$ is a group-like projection.

If $\varphi=\psi$, we have
\begin{eqnarray*}
\varphi(h)\psi(h)&=&(\varphi\otimes\psi)(h\otimes h)\Delta(h)(1\otimes h)\\
&=&(\varphi\otimes\psi)(1\otimes h)\Delta(h)(h\otimes h).
\end{eqnarray*}

By the rignt invariance of $\psi$, we have
$$(\psi\otimes\ h\varphi h)\Delta(h)=(\psi\otimes\varphi)(1\otimes h)\Delta(h)(1\otimes h)=\psi(h)\varphi(h).$$

Since $\varphi=\psi$ and
\begin{eqnarray*}
\lefteqn{(\varphi\otimes\varphi)((1\otimes h)\Delta(h)-h\otimes h)(\Delta(h)(1\otimes h)-h\otimes h)}\\
&=&\varphi(h)^2-\varphi(h)^2-\varphi(h)^2+\varphi(h)^2,
\end{eqnarray*}
we see that $(1\otimes h)\Delta(h)=h\otimes h$ and $h$ is a group-like projection.
\end{proof}

\begin{question}
Is there a biprojection $h$ in a locally compact quantum group which is not a group-like projection?
\end{question}

\begin{proposition}\label{hyglp123}
Let $\mathbb{G}$ be a locally compact quantum group with a group-like projection in $L^1(\mathbb{G})\cap L^\infty(\mathbb{G})$.
Then Young's inequality in Theorem \ref{younglcqg} and Hausdorff-Young inequality in \cite{Cooney} are sharp.
\end{proposition}
\begin{proof}
Suppose $h$ is a group-like projection in $\mathcal{L}_\varphi$.
Since $h$ is analytic with respect to $\sigma^\varphi$ and $\lambda(h\varphi)$ is analytic with respect to $\hat{\sigma}^{\hat{\varphi}}$,
we have that $hd^{1/p}\subseteq d^{1/p}h$ for $1\leq p\leq 2$ and $\lambda(h\varphi)\hat{d}^{1/q}=\hat{d}^{1/q}\lambda(h\varphi)$ for $2\leq q\leq \infty$ from Theorem 2.4 in \cite{Cooney}.
Now
$$\|h\|_p=\|[hd^{1/p}]\|_{p,L^p(\phi)}=\varphi(h)^{1/p}$$
and
$$\|\lambda(h\varphi)\|_q=\|[\lambda(h\varphi)\hat{d}^{1/q}]\|_{q,L^q(\phi)}=\varphi(h)^{\frac{q-1}{q}}.$$
If $\frac{1}{p}+\frac{1}{q}=1$, we have $\|\mathcal{F}_p(h)\|_q=\|h\|_p$ for any $1\leq p\leq 2$.

Since $\frac{\lambda(h\varphi)}{\varphi(h)}$ is a group-like projection, we see that $\hat{\sigma}_t(\lambda(h\varphi))=\lambda(h\varphi)$ and $\hat{\sigma}_t(\lambda(h\varphi))=\lambda(\rho_t(h\varphi))$.
Hence $\rho_t(h\varphi)=h\varphi$ for $t\in\mathbb{R}$ and $\|h\varphi*\rho_{-p'}(h\varphi)\|_r=\|\varphi(h) h\varphi\|_r=\varphi(h)^{1+1/r}$ for $1\leq r\leq 2$.
If $1+\frac{1}{r}=\frac{1}{p}+\frac{1}{q}$, we have $\|h\varphi*\rho_{-p'}(h\varphi)\|_r=\|h\varphi\|_p\|h\varphi\|_q$.
\end{proof}

\begin{definition}
Let $\mathbb{G}$ be a locally compact quantum group. A projection $x$ in $L^\infty(\mathbb{G})$ is a right shift of a group-like projection $h$ if $\psi(x)=\psi(h)$,
$$\Delta(x)(h\otimes 1)=h\otimes x,\quad \Delta(h)(x\otimes 1)=x\otimes R(x).$$
A projection $x$ in $L^\infty(\mathbb{G})$ is a left shift of a group-like projection $h$ if $\varphi(x)=\varphi(h)$ and
$$\Delta(x)(1\otimes h)=x\otimes h,\quad \Delta(h)(1\otimes x)=R(x)\otimes x.$$
\end{definition}

\begin{remark}
If $x$ is a right shift of a group-like projection $h$, then $R(x)$ is a left shift of $h$.
\end{remark}

\begin{remark}
Let $G$ be a locally compact group and $H$ a subgroup of $G$.
Suppose $1_{xH}$ is the characteristic function on a left coset $xH$ of $H$. Then $1_{xH}$ is a left shift of $1_{H}$.
\end{remark}

\begin{proposition}\label{lrshift}
Let $\mathbb{G}$ be a locally compact quantum group.
Suppose $x\in\mathcal{L}_\psi$ is a right shift of a group-like projection $h\in\mathcal{L}_\psi$ and $y\in\mathcal{L}_\varphi$ is a left shift of $h$.
Then
$$\tau_t(x)=x,\quad \sigma_t^\psi(x)=x,\quad x\delta^{it}=\mu_x^{it}x$$
for some $\mu_x>0$ and all $t\in\mathbb{R}$
and
$$\tau_t(y)=y,\quad \sigma_t^\varphi(y)=y,\quad y\delta^{it}=\mu_y^{it}y$$ for some $\mu_y>0$ and all $t\in\mathbb{R}$.
\end{proposition}
\begin{proof}
Suppose $x\in\mathcal{L}_\psi$ is a right shift of $h$. Then
$$\psi(h)x=(\psi\otimes \iota)(\Delta(x)(h\otimes 1)),\quad \psi(x)R(x)=(\psi\otimes \iota)(\Delta(h)(x\otimes 1)).$$
Now we see that $x$ is in $\mathcal{D}(S)$, and
\begin{eqnarray*}
\psi(h)S(x)=(\psi\otimes\iota)((x\otimes 1)\Delta(h))=\psi(x)R(x).
\end{eqnarray*}
Hence $\tau_t(x)=x$ for all $t\in \mathbb{R}$.

By the relation $\Delta\tau_t=(\sigma_t^\varphi\otimes \sigma_{-t}^\psi)\Delta$ in Proposition 6.8 of \cite{KuVaes}, we have
\begin{eqnarray*}
\psi(h)\sigma_{-t}^\psi(x)&=&(\psi\sigma_t^\varphi\otimes\sigma_{-t}^\psi)(\Delta(x)(h\otimes 1))\\
&=&(\psi\otimes\iota)(\Delta(\tau_t(x))(\sigma_t^\varphi(h)\otimes 1))\\
&=&(\psi\otimes\iota)(\Delta(x)(h\otimes 1))=\psi(h)x.
\end{eqnarray*}

By the relation $\Delta(\delta)=\delta\otimes\delta$ in Proposition 7.12 of \cite{KuVaes} and Corollary \ref{mo}, we have
\begin{eqnarray*}
\delta^{it}x\otimes\delta^{it}R(x)&=&(\delta^{it}\otimes \delta^{it})\Delta(h)(x\otimes 1)\\
&=&\Delta(\delta^{it}h)(x\otimes 1)=x\otimes R(x).
\end{eqnarray*}
Therefore for any $\omega\in L^1(\mathbb{G})$, we have $\omega(\delta^{it}R(x))\delta^{it}x=\omega(R(x))x$.
Hence there exists $\mu_x> 0$ such that $\delta^{it}x=\mu_x^{it}x$.

Following the argument above, we have similar properties for a left shift of a group-like projection.
\end{proof}

\begin{corollary}\label{lrc}
Let $\mathbb{G}$ be a locally compact quantum group.
Suppose $x\in\mathcal{L}_\varphi$ is a left shift of a group-like projection $h\in \mathcal{L}_\varphi$.
Then $\hat{\sigma}_t^{\hat{\varphi}}(\mathcal{F}_1(x))=\mu_x^{-it}\mathcal{F}_1(x)$, $\hat{\tau}_t(\mathcal{F}_1(x))=\mathcal{F}_1(x)$ and $\hat{\delta}^{it}\mathcal{F}_1(x)=\mathcal{F}_1(x)$ for all $t\in\mathbb{R}$.
\end{corollary}
\begin{proof}

By Proposition \ref{lrshift}, we have
\begin{eqnarray*}
\hat{\sigma}_t^{\hat{\varphi}}(\mathcal{F}_1(x))=\hat{\sigma}_t^{\hat{\varphi}}(\lambda(x\varphi))
=\lambda(\rho_t(x\varphi))=\mu_x^{-it}\lambda(x\varphi)
\end{eqnarray*}
and
$$\hat{\tau}_t(\mathcal{F}_1(x))=\lambda((x\varphi)\tau_{-t})=\lambda(x\varphi).$$
By the fact that $\sigma_t(x)=x$, we see that $\hat{\delta}^{it}\mathcal{F}_1(x)=\mathcal{F}_1(x)$ for all $t\in\mathbb{R}$.
\end{proof}

\begin{remark}
Suppose $\mathbb{G}$ is a locally compact quantum group such that
the scaling automorphism group $\tau_t$ is trivial and $\varphi=\psi$ is a tracial weight.
Then the assumption $\Delta(x)(h\otimes 1)=h\otimes x$ is equivalent to the assumption $\Delta(h)(x\otimes 1)=x\otimes R(x)$.
\end{remark}

\begin{definition}
Let $\mathbb{G}$ be a locally compact quantum group. A nonzero element $x$ in $\mathcal{L}_\varphi$ is a bi-partial isometry if $x$ and $\mathcal{F}_1(x)$ are multiples of partial isometries.
\end{definition}

\begin{proposition}\label{bipar}
Let $\mathbb{G}=(\mathcal{M},\Delta,\varphi,\psi)$ be a locally compact quantum group.
Suppose $h\in\mathcal{L}_\varphi$ is a group-like projection and $x\in\mathcal{L}_\varphi$ is a left shift of $h$.
Then $x$ is a bi-partial isometry. Moreover $\mathcal{F}_1(x)^*\mathcal{F}_1(x)=\varphi(h)\mathcal{F}_1(h)$ and $\|\mathcal{F}_1(x)\|_\infty=\varphi(h)$.
\end{proposition}
\begin{proof}
Since $x$ is a left shift of $h$, we have $\tau_t(x)=x$, $\sigma_t^\varphi(x)=x$, $\delta^{it}x=\mu_x^{it}x$ for some $\mu_x>0$ and all $t\in\mathbb{R}$.
Hence $\lambda(x\varphi)=\mathcal{F}_1(x)$ satisfyes $\hat{\sigma}^{\hat{\varphi}}_t(\lambda(x\varphi))=\lambda(\rho_t(x\varphi))=\mu_x^{-it}\lambda(x\varphi)$.

Note that $(1\otimes \varphi)(\Delta(h)(1\otimes x))=\varphi(x)R(x)$, we then see that $x*h=\varphi(h)x$ and
$\mathcal{F}_1(x)\mathcal{F}_1(h) = \varphi(h)\mathcal{F}_1(x)$.

We shall show that $x$ is a bi-partial isometry.
For any $b$ in $\mathcal{D}(S)$ such that $S(b)\in\mathfrak{N}_\varphi$,
\begin{eqnarray*}
(x\varphi)^*(b)&=&\overline{x\varphi}(S(b))=\overline{\varphi(S(b)^*x)}=\varphi(xS(b))\\
&=&\varphi S(bS^{-1}(x))=\varphi R(bR(x))=(R(x)\psi)(b)
\end{eqnarray*}
Hence $(x\varphi)^*=R(x)\psi$.
Note that
\begin{eqnarray*}
\lambda(x\varphi)^*\lambda(x\varphi)&=&\lambda(R(x)\psi)\lambda(x\varphi)=\lambda(R(x)\psi* x\varphi).
\end{eqnarray*}
For any $b$ in $\mathcal{R}_\varphi$, we then have
\begin{eqnarray*}
(R(x)\psi* x\varphi)(b)&=&(R(x)\psi\otimes x\varphi)\Delta(b)\\
&=&(R(x)\psi)((\iota\otimes\varphi)(\Delta(b)(1\otimes x)))\\
&=& (R(x)\psi S^{-1})((\iota\otimes\varphi)((1\otimes b)\Delta(x))).
\end{eqnarray*}
For any $a$ in $\mathcal{D}(S^{-1})\cap\mathfrak{N}_\varphi$ such that $S^{-1}(a)\in\mathfrak{N}_\psi^*$, we have
\begin{eqnarray*}
(R(x)\psi S^{-1})(a)&=&\psi(S^{-1}(a)R(x))=\psi S^{-1}(S(R(x))a)=\varphi(xa),
\end{eqnarray*}
i.e. $R(x)\psi S^{-1}=\varphi x= x\varphi$. Now we see that
$R(x)\psi* x\varphi=(x\varphi\otimes \iota)(\Delta(x))\varphi$.
We will show that $(x\varphi\otimes \iota)(\Delta(x))=\varphi(h)h$.
Since $\sigma^\varphi_t(x)=x$ and $x$ is a projection, we see that $(x\varphi\otimes \iota)(\Delta(x))>0$.
By the relation $\Delta(x)(1\otimes h)=x\otimes h$, we have that
$$h(x\varphi\otimes \iota)(\Delta(x))=(x\varphi\otimes \iota)(\Delta(x))h=\varphi(x)h=\varphi(h)h.$$
Therefore $(x\varphi\otimes \iota)(\Delta(x))\geq \varphi(h)h$.
On the other hand, we have
$$\varphi((x\varphi\otimes \iota)(\Delta(x)))=\varphi(x)^2=\varphi(h)^2.$$
Hence $(x\varphi\otimes \iota)(\Delta(x))=\varphi(h)h$ and
$\lambda(x\varphi)^*\lambda(x\varphi)=\varphi(h)\lambda(h\varphi)$. \\
Moreover, $\|\mathcal{F}_1(x)\|_\infty = \varphi(h)$.
\end{proof}

\begin{definition}
Suppose $\mathbb{G}$ is a locally compact quantum group.
An element $x$ in $L^1(\mathbb{G})\cap L^2(\mathbb{G})$ is said to be $p$-extremal if $\|\mathcal{F}_p(x)\|_q=A_p(\mathbb{G})\|x\|_p$, where $p\in[1,2]$, $\frac{1}{p}+\frac{1}{q}$, and $A_p(\mathbb{G})$ is the best constant for the inequality.
A pair $(x,y)$ is said to $(p,q)$-extremal if $\|x*\rho_{-i/p'}(y)\|_r=B_{p,q}(\mathbb{G})\|x\|_p\|y\|_q$, where $x,y\in L^1(\mathbb{G})\cap L^2(\mathbb{G})$, $p,q\in[1,2]$, $\frac{1}{r}+1=\frac{1}{p}+\frac{1}{q}$, and $B_{p,q}(\mathbb{G})$ is the best constant for the inequality.
\end{definition}

\begin{remark}
Suppose $\mathbb{G}$ is a locally compact quantum group with a group -like projection in $\mathcal{L}_\varphi$.
Then, by Proposition \ref{hyglp123}, every group-like projection $h$ in $\mathcal{L}_\varphi$ is $p$-extremal and every pair $(h,h)$ is $(p,q)$-extremal.
\end{remark}

\begin{corollary}
Suppose $\mathbb{G}$ is a locally compact quantum group and $x$ is a left shift of a group-like projection $h$.
Then $x$ is $p$-extremal. If $\varphi=\psi$, we have $(R(x),x)$ is $(p,q)$-extremal.
\end{corollary}
\begin{proof}
Note that $\sigma_t^\varphi(x)=x$ and $\varphi(x)=\varphi(h)$, we have
$$\|x\|_p=\|xd^{1/p}\|_{p,L^p(\phi)}=\varphi(h)^{1/p}=\|h\|_p.$$
By Proposition \ref{bipar}, we have $\|\lambda(x\varphi)\|_q=\|\lambda(h\varphi)\|_q$. Hence $x$ is $p$-extremal.

Note that $R(x)\psi * x\varphi=\varphi(h) h\varphi$, we can see that $\|R(x)*x\|_r=\varphi(h)^{\frac{1}{r}+1}$, $\|R(x)\|_p=\varphi(h)^{1/p}$, $\|x\|_q=\varphi(h)^{1/q}$.
\end{proof}

\begin{definition}
Suppose $\mathbb{G}$ is a locally compact quantum group and $h$ is a group-like projection in $\mathcal{L}_\varphi$, $\tilde{h}$ is the range projection of $\mathcal{F}_1(h)$.
A nonzero element $x$ in $\mathcal{L}_\varphi$ is said to be a bi-shift of a biprojection $h$ if there exist a left shift $x_h$ of $h$ and a left shift $x_{\tilde{h}}$ of $\tilde{h}$ and an element $y\in\mathcal{L}_\varphi$ such that $\tau_t(y)=\sigma_t^\varphi(y)=y$ for all $t\in\mathbb{R}$ and
$$x=(x_hy)*\widehat{\mathcal{F}}_1(x_{\tilde{h}}).$$
\end{definition}

\begin{theorem}
Suppose $\mathbb{G}$ is a locally compact quantum group and $x\in\mathcal{L}_\varphi$ is a bishift of a group-like projection $h$ as above. Then $\sigma^\varphi_t(x)=x$, $\delta^{it}x=\mu_{x_h}^{it}x$ for all $t\in\mathbb{R}$, some $\mu_{x_h}>0$, $x$ is bi-partial isometry and $\|\mathcal{F}_1(x)\|_\infty=\|x\|_1$. Moreover, $x$ is $p$-extremal for $1\leq p\leq 2$.
\end{theorem}
\begin{proof}
By Proposition \ref{lrshift}, we have $\tau_t(x_h)=x_h$ and hence
$$x=(((x_hy)\varphi)R\otimes \iota)\Delta(\widehat{\mathcal{F}}_1(x_{\tilde{h}})).$$
By Corollary \ref{lrc}, we see that
\begin{eqnarray*}
\tau_t(\widehat{\mathcal{F}}_1(x_{\tilde{h}}))=\widehat{\mathcal{F}}_1(x_{\tilde{h}}),\quad
\sigma_t^\varphi(\widehat{\mathcal{F}}_1(x_{\tilde{h}}))=\mu_{x_{\tilde{h}}}^{-it}\widehat{\mathcal{F}}_1(x_{\tilde{h}}),
\end{eqnarray*}
for some $\mu_{x_{\tilde{h}}}>0$ and all $t\in\mathbb{R}$.

From the relation $\Delta\sigma_t^\varphi=(\tau_t\otimes \sigma_t^\varphi)\Delta$, we have
\begin{eqnarray*}
\sigma_t^\varphi(x)&=&((x_hy)\varphi R\otimes\sigma_t^\varphi)\Delta(\widehat{\mathcal{F}}_1(x_{\tilde{h}}))\\
&=&((x_hy)\varphi R\tau_t\otimes\sigma_t^\varphi)\Delta(\widehat{\mathcal{F}}_1(x_{\tilde{h}}))\\
&=&\mu_{x_{\tilde{h}}}^{-it}((x_hy)\varphi R\otimes\iota)\Delta(\widehat{\mathcal{F}}_1(x_{\tilde{h}}))\\
&=&\mu_{x_{\tilde{h}}}^{-it}x,
\end{eqnarray*}
i.e. $\sigma_t^\varphi(x)=\mu_{x_{\tilde{h}}}^{-it}x$ and $\sigma_t^\varphi(x^*x)=x^*x$.

From the relation $\delta^{it}\otimes \delta^{it}=\Delta(\delta^{it})$, we see that
\begin{eqnarray*}
\delta^{it}x&=&((x_hy)\varphi R\otimes \delta^{it})\Delta(\widehat{\mathcal{F}}_1(x_{\tilde{h}}))\\
&=&(((x_hy)\varphi) R\delta^{-it}\otimes \iota)\Delta(\delta^{it}\widehat{\mathcal{F}}_1(x_{\tilde{h}}))\\
&=&\mu_{x_h}^{it}x,
\end{eqnarray*}
for some $\mu_{x_h}>0$ and all $t\in\mathbb{R}$.

Let $q=\mathcal{R}(y^*x_h)$ be the range projection of $y^*x_h$.
Since $\sigma_t^\varphi(y)=y$ and $\sigma_t^\varphi(x_h)=x_h$,
we have $\varphi(q)=\varphi(\mathcal{R}(x_hy))\leq \varphi(x_h)<\infty$. By Lemma 9.5 in \cite{KuVaes}, we have that
\begin{eqnarray*}
\lefteqn{\frac{1}{\varphi(q)}(\iota\otimes\omega_{\Lambda_\varphi(x_hy),\Lambda_\varphi(q)})\Delta(R(\widehat{\mathcal{F}}_1(x_{\tilde{h}})))(\iota\otimes\omega_{\Lambda_\varphi(x_hy),\Lambda_\varphi(q)})\Delta(R(\widehat{\mathcal{F}}_1(x_{\tilde{h}})))^*}\\
&\leq&\iota\otimes\omega_{\Lambda_\varphi(x_hy),\Lambda_\varphi(x_hy)}\Delta(R(\widehat{\mathcal{F}}_1(x_{\tilde{h}}))R(\widehat{\mathcal{F}}_1(x_{\tilde{h}}))^*)
\end{eqnarray*}
On the other hand,
\begin{eqnarray*}
\lefteqn{\frac{1}{\varphi(q)}(\iota\otimes\omega_{\Lambda_\varphi(x_hy),\Lambda_\varphi(q)})\Delta(R(\widehat{\mathcal{F}}_1(x_{\tilde{h}})))(\iota\otimes\omega_{\Lambda_\varphi(x_hy),\Lambda_\varphi(x_h)})\Delta(R(\widehat{\mathcal{F}}_1(x_{\tilde{h}})))^*}\\
&=&\frac{1}{\varphi(q)}((x_hy)\varphi R\otimes R)\Delta(\widehat{\mathcal{F}}_1(x_{\tilde{h}}))(((x_hy)\varphi R\otimes R)\Delta(\widehat{\mathcal{F}}_1(x_{\tilde{h}})))^*\\
&=&\frac{1}{\varphi(q)}R(x)R(x)^*
\end{eqnarray*}
and by Proposition \ref{bipar} and Remark \ref{glpbir}, we can obtain
\begin{eqnarray*}
\lefteqn{\iota\otimes\omega_{\Lambda_\varphi(x_hy),\Lambda_\varphi(x_hy)}\Delta(R(\widehat{\mathcal{F}}_1(x_{\tilde{h}}))R(\widehat{\mathcal{F}}_1(x_{\tilde{h}}))^*)}\\
&=&\hat{\varphi}(\tilde{h})(\iota\otimes (x_hy)(\varphi)(y^*x_h))\Delta(R(\widehat{\mathcal{F}}_1(\tilde{h})))\\
&=&\frac{1}{\varphi(h)^2}(\iota\otimes (x_hy)(\varphi)(y^*x_h))\Delta(h)\\
&=&\frac{1}{\varphi(h)^2}R(x_h)\varphi(y^*x_hy).
\end{eqnarray*}
Therefore
$$x^*x\leq \frac{\varphi(q)\varphi(y^*x_hy)}{\varphi(h)^2}x_h,$$
i.e. $\mathcal{R}(x^*)\leq x_h$.

Note that $\mathcal{F}_1(x)^*=x_{\tilde{h}}\mathcal{F}_1(x_hy)^*$. Then $\mathcal{R}(\mathcal{F}_1(x)^*)\leq x_{\tilde{h}}$.

Since $\sigma_t^\varphi(x^*x)=x^*x$, we obtain $\sigma_t^\varphi(\mathcal{R}(x^*))=\mathcal{R}(x^*)$. By Proposition \ref{glpbi}, we see
\begin{eqnarray*}
\|\mathcal{F}_1(x)\|_\infty&\leq &\|x\|_1=\|x\mathcal{R}(x^*)d\|_{1,L^1(\phi)}\\
&\leq &\|\mathcal{R}(x^*)d^{1/2}\|_{2,L^2(\phi)}\|xd^{1/2}\|_{2,L^2(\phi)}\\
&=&\varphi(\mathcal{R}(x^*))^{1/2}\|\mathcal{F}_1(x)\|_2\\
&=&\varphi(\mathcal{R}(x^*))^{1/2}\|\mathcal{F}_1(x)\mathcal{R}(\mathcal{F}_1(x)^*)\hat{d}^{1/2}\|_{2,L^2(\hat{\phi})}\\
&\leq &\varphi(x_h)^{1/2}\|\mathcal{F}_1(x^*)\|_\infty\hat{\varphi}(\mathcal{R}(\mathcal{F}_1(x)^*))^{1/2}\\
&\leq&\varphi(h)^{1/2}\hat{\varphi}(\tilde{h})^{1/2}\|\mathcal{F}_1(x)\|_\infty=\|\mathcal{F}_1(x)\|_\infty.
\end{eqnarray*}
Hence the inequalities above must be equalities and we have that $x$ is a bi-partial isometry and $\|\mathcal{F}_1(x)\|_\infty=\|x\|_1$.
Now we have that $x$ is $p$-extremal for Hausdorff-Young inequality.
\end{proof}

\begin{ac}
Part of the work was done during multiple of visits of Jinsong Wu to Hebei Normal University and Jinsong Wu was supported by NSFC (Grant no. A010602).
Zhengwei Liu was supported by a grant from Templeton Religion Trust and an AMS-Simons Travel Grant.
\end{ac}


\end{document}